\documentclass{amsart}

\usepackage{amsmath,amsfonts,amssymb}
\usepackage{graphicx,psfrag,subfigure}
\newtheorem{teo}{Theorem}
\newtheorem{rmk}{Remark}
\newtheorem{prop}{Proposition}
\newtheorem{clly}{Corollary}
\newtheorem{lemma}{Lemma}
\newtheorem{dfn}{Definition}

\newcommand{\R}{{\mathbb{R}}}

\newcommand{\Z}{{\mathbb{Z}}}
\newcommand{\N}{{\mathbb{N}}}

\def\blacksquare{\hbox{\vrule width 4pt height 4pt depth 0pt}}
\def\qed{\ \ \ \hbox{}\nolinebreak\hfill $\blacksquare \  \  \  \  $ \par{}\medskip}

\begin{document}
	\title[Top. of global attractors for homeos with TSP.]{Topology of  global attractors for  homeomorphisms with the topological shadowing property in $\R^m$.}
	\author[1]{Gonzalo Cousillas $^1$}	 
	\address[$^1$]{Departamento de Matemática y Aplicaciones, Centro Universitario Regional Este, Universidad de la República, Maldonado, Uruguay.}
	\email{gonzalo.cousillas@cure.edu.uy}
	\thanks{Corresponding author: gonzalo.cousillas@cure.edu.uy}

	\author[]{Jorge Groisman $^2$}
	\address[$^2$]{ Instituto de Matemática y Estadística ``Rafael Laguardia'', Facultad de Ingenier\'{\i}a,  Universidad de la Rep\'ublica, Montevideo, Uruguay.}
	\email{jorgeg@fing.edu.uy}
	
	\maketitle
	
\begin{abstract}  We study the dynamics of  homeomorphisms with the  topological shadowing property and a  global attractor in $\R^m$. 
	We prove that under these hypothesis the attractor is trivial.		
\end{abstract}

\section{Introduction.}

The shadowing property (or pseudo-orbit tracing property) is a cornerstone of modern dynamical systems theory. It provides the essential theoretical justification for numerical simulations: in systems possessing this property, a pseudo-orbit—a sequence of points generated with small, persistent computational errors—is ``shadowed'' by a true, exact trajectory of the system. Consequently, the qualitative behavior observed in simulations remains a true representation of the system's underlying dynamics.

The concept originated in the 1960s with the work of Anosov \cite{An67}, who demonstrated its role in the structural stability of certain diffeomorphisms. Anosov established that in hyperbolic systems, where the tangent bundle admits a continuous splitting into stable and unstable subspaces, every pseudo-orbit stays remarkably close to a unique physical orbit.

In 1975, Bowen \cite{Bo75} transformed this observation into a fundamental tool for the study of Axiom A diffeomorphisms. Using the shadowing property, Bowen showed that complex continuous systems could be modeled via simpler symbolic dynamics (such as subshifts of finite type), establishing shadowing as a subject of intrinsic interest. Subsequently, Walters \cite{Wa78} redirected the focus toward the interplay between shadowing and stability. A key requirement in this context is expansivity—the property that any two distinct orbits eventually separate by a fixed distance. Walters proved that for an expansive homeomorphism on a compact metric space, the shadowing property implies topological stability, ensuring the system's qualitative features remain invariant under small perturbations.



In the present paper, we examine shadowing within a topological setting where the map is a homeomorphism and the ambient space is non-compact. In the absence of the regularity typically in a compact context, we adopt the generalized definition of  shadowing proposed in \cite{DLRW13} for non-metrizable spaces, or equivalently,  for non-compact metric spaces as defined in \cite{LNY18}. This is called {\it topological shadowing property}. These works also introduce {\it topological expansivity}, a notion that also generalizes the classical definition while remaining—both of these—invariant under topological conjugacy.

Recent literature has utilized these definitions to extend classical results to non-compact domains. Notably, versions of Smale’s Spectral Decomposition Theorem and Walters’ Stability Theorem have been established for ``topologically Anosov'' homeomorphisms—those that are both topologically expansive and possess the topological shadowing property. Furthermore, in \cite{CGX21} provided a classification for such homeomorphisms on surfaces of genus zero and finite type, proving they are conjugate to either a homothety or a reverse homothety on the plane.

In this work, we diverge from the assumption of topological expansivity. Instead, we assume the existence of a compact global attractor. 
We investigate the topological structure of the global attractor for homeomorphisms defined on $\R^m$. 
Our approach demonstrates that the presence of a global attractor, even in the absence of topological expansivity, imposes rigid constraints on the system's configuration. In particular, we characterize the relationship between the topological shadowing property and the shape of the attractor, culminating in the main result:

\begin{teo}[\bf Main Theorem]\label{K_trivial}   Let $f$ be a homeomorphism of $\R^m$ ($m\geq 2$) with the topological shadowing property. If $f$ possesses a compact global attractor $K$, then $K$ must be trivial.
\end{teo}

The remainder of this paper is organized as follows. In Section \ref{prel}, we establish the Preliminaries, providing the foundational definitions for the topological shadowing property and reviewing essential results from the existing literature on global attractors. Section \ref{at_trivial} explores the interplay between global attractors and the topological shadowing property, specifically how the existence of an attractor restricts  the behavior of orbits near the attractor.
In Section \ref{main_thm}, we provide conditions for the triviality of the global attractor, demonstrating that under certain topological constraints, the attractor must reduce to a single fixed point. Finally, Section \ref{one_dim}  is dedicated to the One-Dimensional Case; here, we provide a specific counterexample.

\section{Preliminaries}\label{prel}
	
	\subsection{Basic definitions and previous results}
	
	Let $(X,d)$ be a metric space and let $A\subset X$. We denote by $\overline{A}$ the closure of $A$.  
	Let $f:X\to X$ be a homeomorphism. A point $x\in X$ is said to be \emph{Lyapunov stable} if for every $\epsilon>0$ there exists $\delta>0$ such that $d(x,y)<\delta$ implies 
	\[
	d(f^n(y), f^n(x))<\epsilon \quad \text{for all } n\in\N.
	\]
	It is easy to see, for instance, that under a homothethy of $\R^m$ every point is Lyapunov stable. We denote $\mathcal{O}(x)=\{f^n(x): n\in \Z\}$ the orbit of $x$, $\mathcal{O}^+(x)=\{f^n(x): n\in \Z^+\}$ the forward orbit and $\mathcal{O}^-(x)=\{f^n(x): n\in \Z^-\}$ the backward orbit.
	
	The \emph{$\omega$-limit set} of a point $x\in X$ is defined by
	\[
	\omega(x)=\bigcap_{n=0}^\infty \overline{\{f^k(x): k\geq n\}}.
	\]
	The \emph{$\alpha$-limit set} of $x$ is the $\omega$-limit set of $x$ with respect to $f^{-1}$.
	
	

	We now introduce the concept of attractor that will be used throughout this paper. Our goal is to show that if $K\subset \R^m$ (with $m\geq2$) is a global attractor, then $\R^m\setminus K$ has exactly one connected component.
	
	\begin{dfn}\label{def_atractor} 
		Let $f:\R^m\to \R^m$ be a homeomorphism and $K\subset \R^m$ a compact subset.
		We say that $K$ is an \emph{ attractor} if there exists a bounded open neighborhood $U$ of $K$ such that $\overline{f(U)}\subset U$ and
		\[
		K=\bigcap_{n>0} f^n(U).
		\]
		The \emph{basin of attraction} of $K$ is the set $B=\bigcup_{n\geq0} f^{-n}(U)$.  
		If $K$ is an attractor and moreover $B=\R^m$, then $K$ is called a \emph{ global attractor}.
	\end{dfn}
	
	Note that every  attractor $K$ is invariant. If $K$ is a  global attractor, then  $\{f^n(U)\}_{n\in\N}$ is a nested sequence of sets for  $U$ as in Definition~\ref{def_atractor}. Setting $A_n=\overline{f^n(U)}\setminus \overline{f^{n+1}(U)}$, we obtain the disjoint union
	\[
	\R^m=\left(\bigcup_{n\in \Z} A_n\right)\cup K.
	\]
	
	Furthermore, it is straightforward to verify that under this definition,  $K$ attracts bounded sets. That is,  for every compact set $C$  and every  $\epsilon>0$, there exists $n_0\in \N$ such that $f^{n}(C)\subset B(K,\epsilon)$ for all $n\geq n_0$, where $B(K, \epsilon)=\bigcup_{x\in K} B(x,\epsilon)$. 
	
	\begin{rmk}\label{K_sin_agujeros}
		If $K$ is a  global attractor, then: 
	\begin{enumerate} \item $\omega(x)\subset K$ for every $x\in \R^m$.
		 \item $\alpha(x)=\emptyset$ for every $x\in \R^m\setminus K$.
	\end{enumerate}	 
	\end{rmk} 
	
	\begin{proof}$\,$

	\begin{enumerate}
\item		Let $x\in \R^m\setminus K$. Since $\overline{\mathcal{O}^+(x)}$ is bounded, then  $\omega(x)\subset \overline{\mathcal{O}^+(x)}$. Since  $K$ attracts bounded sets, it attracts $\omega(x)$ as well. By the definition of $\omega(x)$, it is not hard to conclude that $\omega(x)\subset K$.  
		
	\item 	Now suppose that $\alpha(x)\neq \emptyset$ for some $x\in \R^m\setminus K$. Let $y\in\alpha(x)$. Then $y\notin K$, therefore $y\in A_{n_0}$ for some $n_0$. Since $\alpha(x)$ is closed and invariant, $\omega(y)\subset\alpha(x)$. This contradicts the fact that $\omega(y)\subset K$. Hence $\alpha(x)=\emptyset$.
	\end{enumerate}
	\end{proof}
	
	In case $K=\{p\}$ for some $p\in \R^m$, we denote $K$ a \emph{trivial  attractor}.  
	
	The following lemma, due to Hale, establishes connectedness of global attractors in Banach spaces:
	
	\begin{lemma}\cite{Ha88}\label{K_conexo}
		If $X$ is a Banach space and $K$ is a compact invariant set which attracts compact sets, then $K$ is connected.
	\end{lemma}
	
	Therefore a  global attractor $K$ in $\R^m$ is connected.  
	
	\begin{rmk}\label{no_agujeros} $\R^m\setminus K$ has no bounded connected components.
	\end{rmk}
		Indeed, suppose $C=\bigcup C_\lambda$ is the union of bounded components of $\R^m\setminus K$ and $y\in C$, then the backward orbit of $y$ is contained in $C$ and hence it is bounded. This  implies $\alpha(y)\neq \emptyset$, contradicting Remark~\ref{K_sin_agujeros}.

	
	
	We now prove that if $m\geq2$, then $\partial K$ is connected. For this we use the following result:
	
	\begin{prop}\label{interseccion_comp_conexo}
	Let $X$ be a compact metric space and let $\{E_n\}_{n\in \N}$ be a sequence of nonempty compact connected sets of $X$.  
		Let $\{x_n\}_{n\in \N}$ be a sequence with $x_n\in E_n$ for all $n$ such that  $x_n\to x$.  
		Then
		\[
		E=\bigcap_{j=1}^{+\infty} \overline{\bigcup_{n=j}^{+\infty} E_n}
		\]
		is a nonempty compact connected set containing $x$.
	\end{prop}
	
	\begin{lemma}\label{partial_K_conexo} 
		Let $f:\R^m\to \R^m$ be a homeomorphism with $m\geq 2$ and let $K\subset\R^m$ be a  global attractor. Then $\partial K$ is connected.
	\end{lemma}
	
	\begin{proof}
		By Lemma~\ref{K_conexo}, $K$ is connected  and by Remark~\ref{no_agujeros}, $\R^m\setminus K$ has no bounded components.  
		Let $U$ be the neighborhood from Definition~\ref{def_atractor} and denote by $B_0$ the unbounded component of $\R^m\setminus U$. Set $E_0=\partial B_0$. Then
		\[
		\partial K=\bigcap_{j=1}^{+\infty} \overline{\bigcup_{n=j}^{+\infty} f^n(E_0)}.
		\]
		Setting  $E_n=f^n(E_0)$, then $E_n$ is  a compact and connected set. By Proposition~\ref{interseccion_comp_conexo} it follows that $\partial K$ is connected.
	\end{proof}


	

	

	
	\section{Global attractors and the topological shadowing property}\label{at_trivial}
	
	In this section we investigate how the topological shadowing property constraints the dynamics near a global attractor.  
	We will show that if $f:\R^m\to \R^m$ is a homeomorphism with the topological shadowing property and $K\subset \R^m$ is a global attractor, then every boundary point of $K$ is Lyapunov stable.  
	
	
	We begin with a lemma that ensures that backward iterates of distinct points eventually separate at a definite scale. In other words, there exists $\epsilon\in C^+(\R^m)$ such that $f$ is $\epsilon$-expansive in $\R^m\setminus K$.
	
	\begin{lemma}\label{epsilon_expansivo_gral}
		Let $f:\R^m\to \R^m$ be a homeomorphism and let $K\subset \R^m$ be a global attractor.  
		There exists $\epsilon\in C^+(\R^m)$ such that for every $x\in \R^m\setminus K$ and every $y\neq x$, there exists $n_0\in \N$ with  
		\[
		d\big(f^{-n_0}(x), f^{-n_0}(y)\big)>\epsilon(f^{-n_0}(x)).
		\]
	\end{lemma}
	
	\begin{proof}
		Let $U$ be the neighborhood of $K$ given by the definition of global attractor, and define  
		\[
		A_n=\overline{f^{-n}(U)\setminus f^{-(n-1)}(U)},\quad n\in\N.
		\]
		Each $A_n$ is compact.

		Fix $\epsilon_0>0$ and   $x\in A_0$.  
		Choose $\epsilon_1^x>0$ for $f^{-1}(x)\in A_1$ such that  
		\[
		f\big(B(f^{-1}(x), \epsilon_1^x)\big)\subset B(x,\tfrac{\epsilon_0}{2}).
		\] 
		
		Set  
		\[
		\epsilon_1=\min\{\epsilon_1^{x}>0: x\in A_0\}.
		\]  
		Thus for every $x\in A_0$,  
		\[
		f\big(B(f^{-1}(x), \epsilon_1)\big)\subset B(x,\tfrac{\epsilon_0}{2}).
		\]  
		This defines $\epsilon_1$ for every point in $A_1$.

		Fix some $n\geq 1$ and assume that $\epsilon_i>0$ has been defined for $i\leq n$ so that for every $x\in A_n$,  
		\[
		f^i\big(B(x,\epsilon_n)\big)\subset B\!\left(f^i(x), \tfrac{\epsilon_{n-i}}{2^i}\right),\quad 0\leq i\leq n.
		\]  
		Let $x\in A_n$,  
		choose $\epsilon_{n+1}^x>0$ such that for all $i\leq n+1$,  
		\[
		f^i\big(B(f^{-1}(x),\epsilon_{n+1}^x)\big)\subset B\!\left(f^{i-1}(x), \tfrac{\epsilon_{n+1-i}}{2^i}\right).
		\]  
		Compactness of $A_n$ yields   
		\[
		\epsilon_{n+1}=\min\{\epsilon_{n+1}^{x}: x\in A_n\}>0.
		\]

		Take $x\in A_0$ and $y\in B(x,\epsilon_0)$.  
		By the previous construction there exists $n_0$ such that  
		\[
		y\notin B\!\left(x,\tfrac{\epsilon_0}{2^{n_0}}\right),
		\]
		which implies  
		\[
		d(f^{-n_0}(x), f^{-n_0}(y))>\epsilon_{n_0}.
		\]

		Define $\epsilon:\R^m\to \R^+$ by  
		\[
		\epsilon|_U \equiv \epsilon_0, 
		\qquad \epsilon(x)<\epsilon_i \ \text{for } x\in A_i,\ i\geq 1.
		\]
		This function satisfies the desired property and the result follows.
	\end{proof}

	
	The previous lemma ensures separation of backward orbits. Combined with shadowing, this implies uniqueness of  tracing orbit outside the attractor.
	
	\begin{clly}\label{cor:epsilon_expansivo}
		Let $f:\R^m\to \R^m$ be a homeomorphism with the topological shadowing property, and let $K\subset \R^m$ be a global attractor.  
		There exists $\epsilon\in C^+(\R^m)$ such that, for the corresponding $\delta\in C^+(\R^m)$ from the shadowing property, every $\delta$-pseudo orbit that coincides with the past orbit of some $x\in \R^m\setminus K$ is $\epsilon$-shadowed uniquely by the orbit of $x$.
	\end{clly}
	
	\begin{proof}
		Take $\epsilon\in C^+(\R^m)$ from Lemma \ref{epsilon_expansivo_gral}, and let $\delta\in C^+(\R^m)$ of the topological shadowing property corresponding to $\epsilon$.  
		Consider a $\delta$-pseudo orbit $\{x_n\}_{n\in \Z}$ with $x_n=f^n(x)$ for $n\leq0$ and $x\in \R^m\setminus K$.  
		By Lemma \ref{epsilon_expansivo_gral}, if $y$ satisfies  
		\[
		d(f^n(y),x_n)<\epsilon(x_n),\quad \forall n\leq0,
		\]  
		then $y=x$.  
		Hence the orbit of $x$ is the unique orbit that $\epsilon$-shadows $\{x_n\}_{n\in \Z}$.
	\end{proof}

	
	We can now establish the key stability property for the boundary points of the global attractor.
	
	\begin{prop}\label{Lyap_est}
		Let $f:\R^m\to \R^m$ be a homeomorphism with the topological shadowing property, and let $K\subset \R^m$ be a global attractor.  
		Then every point of $\partial K$ is Lyapunov stable.
	\end{prop}
	
	\begin{proof}
	Let $z\in \partial K$ and  $U$ of the definition of attractor for  $K$.
	Given $\epsilon>0$, consider  $\tilde{\epsilon}\in C^+(\R^m)$ from  Lemma \ref{epsilon_expansivo_gral} such that $\tilde{\epsilon}(x)=\epsilon/2$ for every $x\in U$. Take $\tilde{\delta}\in C^+(\R^m)$ of the shadowing property corresponding to  $\tilde{\epsilon}$.
	Now consider $x\in U\setminus K$ close enough to $z$ satisfying $z\in B(x, \tilde \delta(x))$. 
	Note that there exists $\delta>0$ such that $B(z, \delta)\subset B(x,\tilde{\delta}(x))$. 
	
	Take the following $\tilde{\delta}$-pseudo orbits: 	$\{x_n\}_{n\in\Z}$ such that $$x_n=\begin{cases} f^n(x) &n<  0\\ f^{n}(z) & n\geq 0 \end{cases},$$  and $\{y_n\}_{n\in \Z}$ such that  $$y_n=\begin{cases} f^n(x) & n<0\\  f^{n}(y) &n\geq0\end{cases}$$ for  $y\in B(z,\delta)$. 
	This two $\tilde{\delta}$-pseudo orbits are $\tilde{\epsilon}$-shadowed by the orbit of $x$ due to Corollary  \ref{cor:epsilon_expansivo}. Then 
	\begin{eqnarray*} d(f^{n}(z), f^{n}(y))&\leq &d(f^n(z), f^{n}(x))+d(f^{n}(x), f^{n}(y))\\&\leq&\tilde{\epsilon}(f^n(z))+\tilde{\epsilon}(f^n(y))\\&<&\epsilon/2+\epsilon/2\\ &=&\epsilon\end{eqnarray*}  for every $n>0$.  Then  the result follows.
	\end{proof}
	
	The previous proposition  shows that the shadowing property enforces strong stability restrictions on the boundary of a global attractor.  
	In particular it establishes that every boundary point is Lyapunov stable, a phenomenon that rules out the kind of complex dynamics often observed on attractor boundaries in more general systems.
	
	This rigidity has a decisive consequence: once the attractor’s boundary is entirely Lyapunov stable, the attractor itself must collapse to a trivial form.  
	The next section is devoted to proving this statement precisely, culminating in Theorem \ref{K_trivial}. 
	
	\section{Triviality of the global attractor}\label{main_thm}
	
	We are now ready to prove the main theorem of this paper.  
	The key input is the rigidity obtained in the previous section: every boundary point of a global attractor for a homeomorphism with the topological shadowing property is Lyapunov stable (Proposition \ref{Lyap_est}).  
	This strong form of stability severely restricts the possible geometry of the attractor, ultimately forcing it to be trivial.
	
	The following is a nice result that relates the stability of points of a compact invariant set with the existence of a metric such that the map turns out to be a weak contraction.

	\begin{lemma} \cite[Theorem 2.1]{AAB96}\label{contr_debil}
		Let $(X,d)$ be a compact metric space and  $f:X\to X$ a continuous map. 
		Then the following statements are equivalent.
		\begin{enumerate}
			\item Every $x\in X$ is Lyapunov stable.
			\item The metric $d_f(x,y)=\sup_{n\in\N} d(f^n(x), f^n(y))$ is equivalent to   $d$ and $f$ is a weak  contraction respect to  $d_f$.
		\end{enumerate}
	\end{lemma}	
	
	
	


	The next results intends to show the interaction between the dynamics of points with empty $\alpha$-limit and the dynamics of points of the attractor's boundary.
	First we will show that two different point of $\partial K$ can be accessible by pseudo orbits  that share the same orbit  for backwards iterates.
	Afterwards, from this fact we will show  that every pair of points in $\partial K$ are asymptotic. 

	\begin{lemma}\label{po_conecta}
		Let   $f:\R^m\to \R^m$ ($m\geq2$) be a  homeomorphism with the topological shadowing property and   $K\subset \R^m$ be a  global attractor. 
		For every  $y, z\in \partial K$, $x\neq y$ and $\delta>0$,   there exists a $\delta$-pseudo orbit coming from inifinity passing through $y$ and $f^n(z)$ for some $n\in \N$.
	\end{lemma}	
	
	\begin{proof}	
		Since $\partial K$ is a compact invariant set, by Lemma \ref{contr_debil} there exists a metric $d_f$ equivalent to $d$ such that $f|_{\partial K}$ is a weak contraction.  Take $\delta\in C^+(\R^m)$ and set $\delta_0=\min\{\delta(x): x\in \partial K\}$. Let $y, z\in \partial K$, $y\neq z$,  by Lemma \ref{partial_K_conexo} $\partial K$ is a connected set, then chain connected. Consider a $\delta_0$-chain $\{y=x_0, x_1,\ldots, x_n=z\}\subset \partial K$. 
		
		Take $\{\tilde{x}_i\}_{i\in \Z}$ such that $$\tilde{x}_i=\begin{cases} f^i(x)&\mbox{if } i<0\\f^i(x_i)&\mbox{if } 0\leq i<n\\ f^i(z)& \mbox{if } i\geq n\end{cases}.$$ Let us see that 
		$\{\tilde{x}_i\}_{i\in \Z}$ is a $\delta_0$-pseudo orbit with respect  to the metric $d$. For $i<0$ and $i\geq n$ this is immediate. For $0\leq i<n$ we have 	$$d(f(\tilde{x}_i), \tilde{x}_{i+1})\leq d_f(f(\tilde{x}_i), \tilde{x}_{i+1})=d_f(f^{i+1}(x_i), f^{i+1}(x_{i+1}))\leq d_f(x_i, x_{i+1})<\delta_0.$$ Then the result follows.		
	\end{proof}

	\begin{lemma}\label{asymptotic_points}
		Let $f:\R^m\to \R^m$ ($m\geq 2$) be a homeomorphism with the topological shadowing property and $K\subset \R^m$ be a  global attractor. 
		Then every pair of points in $\partial K$ are asymptotic, i.e. given $y, z\in \partial K$ and $\epsilon>0$, there exists $n_0\in\N$ such that $d(f^k(y), f^k(z))<\epsilon$ for every $k\geq n_0$.
	\end{lemma}
	\begin{proof}
		Take  $d_f$ of Lemma \ref{contr_debil} for $\partial K$.  Given $y, z\in\partial K$ and $\epsilon>0$.
		Take $\tilde{\epsilon}\in C^+(\R^m)$ of Corollary \ref{cor:epsilon_expansivo} such that $\tilde{\epsilon}(x)=\epsilon/2$ for every $x\in U$, where $U$ is the bounded and open   neighborhood of the definition of global  attractor for $K$.
		Let $\tilde{\delta}\in C^+(\R^m)$  of the topological shadowing  corresponding to $\tilde{\epsilon}$.
		Choose $0<\delta<\min\{\delta_0\}\bigcup \{\tilde{\delta}(x):x\in U\}$, where $\delta_0>0$ is found by Lemma \ref{po_conecta}.
		
		For this specific $\delta$, consider the $\delta$-chain 	 $\{y=x_0, x_1,\ldots, x_{n_0}=z\}\subset \partial K$. By Theorem \ref{Lyap_est} every point of $\partial K$ is Lyapunov stable. Take $\alpha>0$ from the stability corresponding to $\delta>0$ and $x\in U\setminus K$ such that $d(x,y)<\alpha$. Note that $d_f(x,y)<\delta$. 
		Consider the following $\delta$-pseudo orbits 	 $\{\tilde{x}_k\}_{k\in\Z}$ and $\{\tilde{y}_k\}_{k\in \Z}$ that verifies Lemma \ref{po_conecta}:
		$$\tilde{x}_k=\begin{cases} f^k(x)&\mbox{ if } k<0\\f^k(x_k)&\mbox{ if } 0\leq k<n_0\\ f^k(z)& \mbox{if } k\geq n_0\end{cases},$$ and  $$\tilde{y}_k=\begin{cases}  f^k(x)&\mbox{ if } k<0\\ f^k(y)& \mbox{ if }k\geq0\end{cases}.$$
		Note that this both $\delta$-pseudo orbits are in fact $\tilde{\delta}$-pseudo orbits.  By Lemma  \ref{epsilon_expansivo_gral}, both $\tilde{\delta}$-pseudo orbits are $\tilde{\epsilon}$-shadowed by the orbit of $x$.
		Then $$d(f^k(y), f^k(z))\leq d_f(f^k(y), f^k(z))\leq d_f(f^k(y), f^k(x))+d_f(f^k(x),f^k(z))<\epsilon$$ for $k\geq n_0$ and the result follows.
	\end{proof}	
	
	We are now in position to stablish that $\partial K$ is trivial. 
	
	\begin{teo}\label{borde_trivial}
		Let $f:\R^m\to \R^m$ ($m\geq 2$) be a homeomorphism with the topological shadowing property. 
		If $K\subset \R^m$ is a global attractor, then $\partial K$ is a trivial set.
	\end{teo}
	
	\begin{proof}
		Take the metric $d_f$ such that $f|_{\partial K}$ is a weak contraction.
		Suppose that $\partial K$ is nontrivial, then $\operatorname{diam}(\partial K)>0$. 
		Since $\partial K$ is compact, there exist $x, y\in \partial K$ such that $\operatorname{diam}(\partial K)=d_f(x,y)$.
		Notice that $d_f(f^{-k}(x), f^{-k}(y))=d_f(x,y)$ for every $k>0$.
		
		Let $\tilde{x}\in \alpha(x)$, there exists a sequence $k_i\to +\infty$ such that $\{f^{-k_i}(x)\}_{i\in \N}$ converges to $\tilde{x}$.
		Assume that $\{ f^{-k_i}(y)\}_{i\in\N}$ converges to some $\tilde{y}\in\alpha(y)$. 
		Then $d_f(\tilde{x}, \tilde{y})=d_f(x,y)$.
		We claim that $\tilde{x}$ and $\tilde{y}$ are not asymptotic: 
		Let $\epsilon>0$ be such that $\epsilon<d_f(x,y)$.
		For every $n>0$, we have
		\begin{eqnarray*}
			d_f(f^{n}(\tilde{x}), f^{n}(\tilde{y})) 
			&=& d_f(f^{n}(\lim_{k_i\to+\infty} f^{-k_i}(x)), f^{n}(\lim_{k_i\to+\infty} f^{-k_i}(y))) \\
			&=& \lim_{k_i\to+\infty} d_f(f^{n-k_i}(x), f^{n-k_i}(y)) \\
			&=& d_f(x,y) \\
			&>& \epsilon.
		\end{eqnarray*}
		This contradicts Lemma \ref{asymptotic_points}. 
		Therefore, $\partial K$ is trivial. 
	\end{proof}
	
	With these preliminaries in place, we  proceed to the proof of our main theorem. \\
	
	\noindent{\bf Proof of Theorem \ref{K_trivial}.}
	By Theorem \ref{borde_trivial}, the set $\partial K$ is trivial; denote $\partial K=\{p\}$. 
	Suppose that $K\neq \partial K$; then $\operatorname{int}(K)\neq \emptyset$. 
	Let $V$ be an open neighborhood of $p$, and choose points $q,r\in V$ such that $q\in \operatorname{int}(K)$ and $r\in \operatorname{int}(K^c)$. 
	Now take two distinct arcs $\gamma$ and $\gamma'$ such that $\gamma\cap\gamma'=\{q,r\}$. 
	This implies that $\gamma\cap\partial K\neq\emptyset$ and $\gamma'\cap\partial K\neq\emptyset$, but $\gamma\cap\gamma'\cap\partial K=\emptyset$. 
	Hence $\partial K$ is not trivial. A contradiction. \qed

We now derive two corollaries that follow directly from the specific topological structure of the global attractor.	

	\begin{clly}
		Let $f:\R^m\to \R^m$ ($m\geq 2$) be a homeomorphism with the topological shadowing property and a global attractor. 
		Then $f$ is topologically expansive.
	\end{clly}
	
	\begin{proof}
		By Theorem \ref{K_trivial}, there exists $p\in\R^m$ such that $K=\{p\}$. 
		By Lemma \ref{epsilon_expansivo_gral}, there exists $\epsilon\in C^+(\R^m)$ such that whenever $x\in \R^m\setminus\{p\}$, $x\neq y$, there exists $n\in \N$ with 
		$d(f^{-n}(x), f^{-n}(y))>\epsilon(f^{-n}(x))$. 
		The result follows.
	\end{proof}
	
	
A fundamental problem in the study of dynamical systems is the classification of homeomorphisms under topological conjugacy. For systems possessing a global attractor and the topological shadowing property, the previous theorem provides a framework to characterize these conjugacy classes.

Historically, this classification began in the 1930s when Kerékjártó \cite{Ke34} established that a planar homeomorphism possessing an asymptotically stable fixed point is topologically conjugate, on its basin of attraction, to either $z\mapsto z/2$ or $z\mapsto \overline{z}/2$ depending on whether the map preserves or reverses orientation. Subsequently, Husch \cite{Hu71} provided necessary and sufficient conditions for a homeomorphism of $\R^m$ to be conjugate to  $x\mapsto 
x/2$  for $m\neq4,5$. 

Building upon these classical results, we now establish the classification of conjugacy classes in the presence of a compact global attractor and the topological shadowing property. Specifically, we obtain the following:
	\begin{clly}
		Let $f:\R^m\to \R^m$ ($m\neq 1,4,5$) be a orientation preserving homeomorphism with the topological shadowing property and a global attractor. 
		Then $f$ is conjugate to  $x\mapsto x/2$.
	\end{clly}
	

	
	\section{The one-dimensional case}\label{one_dim}
	In the one-dimensional case, the result does not holds. 
	By Lemma \ref{K_conexo}, the attractor $K\subset \R$ is connected and hence an interval. 
	If $K$ is not trivial, then $\partial K$ is disconnected, this is the key difference with higher dimensions. 
	In \cite{PeVa91}, a characterization of homeomorphisms with the shadowing property on the interval is given, which we recall below.
	Denote by $\mathcal{F}$ the set of fixed points of $(0,1)$ and $\mathcal{C}=\{ x\in \mathcal{F}: \text{for every } \epsilon>0, \text{ there exist } y,z\in(x-\epsilon,x+\epsilon) \text{ such that } f(y)<y \text{ and } f(z)>z\}.$
	
	\begin{teo}\label{sombra_dim_1}
		Let $f:[0,1]\to[0,1]$ be a continuous increasing function. 
		Then $f$ has the shadowing property if and only if $\mathcal{F}=\mathcal{C}$.
	\end{teo}
	
	\begin{figure}[h!]
		\centering
		\includegraphics[width=0.8\linewidth]{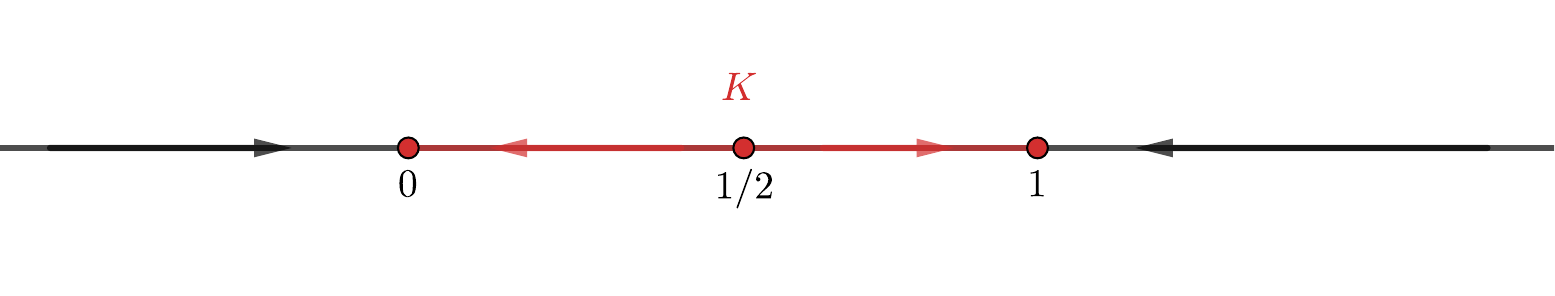}
		\caption[]{Nontrivial global attractor for a homeomorphism on the line.}
		\label{fig:at_intervalo}
	\end{figure}
	
	Consider the homeomorphism $f:\R\to \R$ such that $\operatorname{Fix}(f)=\{0,1/2,1\}$, where $0$ and $1$ are attracting and $1/2$ is repelling (see Figure \ref{fig:at_intervalo}). 
	Moreover, on $(-\infty,0]$, $f$ is a homothety of ratio $1/2$ centered at $0$, and on $[1,+\infty)$ it is a homothety of ratio $1/2$ centered at $1$. 
	
	By Theorem \ref{sombra_dim_1}, the restriction $f|_{[0,1]}$ has the shadowing property, in fact, $[0,1]$ is a nontrivial global attractor. 
	Since $f$ is a homothety on $(-\infty,0]$ and $[1,+\infty)$, it has topological shadowing on these regions as well. 
	Therefore, $f$ has topological shadowing property on $\R$.

	\bibliography{biblio}{}
	
	\bibliographystyle{plain}

\end{document}